\documentclass[12pt]{amsart}
\usepackage{amscd,verbatim}
\usepackage{amssymb}
\usepackage[colorlinks,linkcolor=blue,citecolor=blue,urlcolor=red]{hyperref}

\newcommand{\F}{\mathbf{F}}
\newcommand{\Q}{\mathbf{Q}}
\newcommand{\Z}{\mathbf{Z}}

\newcommand{\G}{\mathbb{G}}

\newcommand{\sA}{\mathcal{A}}
\newcommand{\sC}{\mathcal{C}}

\newcommand{\sH}{\mathcal{H}}
\newcommand{\sK}{\mathcal{K}}
\newcommand{\sM}{\mathcal{M}}

\renewcommand{\L}{\mathbb{L}}

\newcommand{\by}[1]{\overset{#1}{\longrightarrow}}
\newcommand{\iso}{\by{\sim}}

\renewcommand{\lim}{\varprojlim}
\newcommand{\colim}{\varinjlim}

\newcommand{\et}{{\operatorname{\acute{e}t}}}

\newcommand{\tr}{{\operatorname{tr}}}

\newcommand{\ind}{{\operatorname{ind}}}
\newcommand{\rat}{{\operatorname{rat}}}

\newcommand{\eff}{{\operatorname{eff}}}

\newcommand{\alg}{{\operatorname{alg}}}

\newcommand{\op}{{\operatorname{op}}}

\newcommand{\tors}{{\operatorname{tors}}}
\newcommand{\car}{\operatorname{char}}

\newcommand{\Pic}{\operatorname{Pic}}

\newcommand{\Tor}{\operatorname{Tor}}
\newcommand{\Hom}{\operatorname{Hom}}

\newcommand{\Alb}{\operatorname{Alb}}
\newcommand{\Ker}{\operatorname{Ker}}
\newcommand{\IM}{\operatorname{Im}}
\newcommand{\Coker}{\operatorname{Coker}}

\newcommand{\Ab}{\operatorname{\mathbf{Ab}}}
\renewcommand{\Vec}{\operatorname{\mathbf{Vec}}}
\newcommand{\Br}{\operatorname{Br}}
\newcommand{\NS}{\operatorname{NS}}

\newcounter{spec}
{\end{list}}

\newtheorem{thm}{Theorem}
\newtheorem{prop}{Proposition}
\newtheorem{lemme}{Lemma}
\newtheorem{cor}{Corollary}
\theoremstyle{definition}

\newtheorem{defn}{Definition}
\theoremstyle{remark}
\newtheorem{rque}{Remark}
\newtheorem{qn}{Question}

\begin{document}

\title{The Brauer group and  indecomposable $(2,1)$-cycles}
\author{Bruno Kahn}
\address{IMJ-PRG\\Case 247\\4, place Jussieu\\75252 Paris Cedex 05\\France}
\email{bruno.kahn@imj-prg.fr}
\date{July 9, 2015}
\begin{abstract}We show that the torsion in the group of indecomposable $(2,1)$-cycles on a smooth projective variety over an algebraically closed field is isomorphic to a twist of its Brauer group, away from the characteristic. In particular, this group is infinite as soon as $b_2-\rho>0$. We derive a new insight into Ro\v\i t\-man's theorem on torsion $0$-cycles over a surface.
\end{abstract}
\maketitle

\section*{Introduction}

Let $X$ be a smooth projective variety over an algebraically
closed field $k$. The group
\[C(X)=H^1(X,\sK_2)\simeq CH^2(X,1)\simeq H^3(X,\Z(2))\]
has been widely studied. Its most interesting part is the \emph{indecomposable quotient}
\[H_\ind^1(X,\sK_2)\simeq CH_\ind^2(X,1)\simeq H_\ind^3(X,\Z(2))\]
defined as the cokernel of the natural homomorphism 
\begin{equation}\label{eqtheta}
\Pic(X)\otimes k^*\by{\theta} C(X).
\end{equation}

It vanishes for $\dim X\le 1$. 

Let $\Br(X)=H^2_\et(X,\G_m)$ be the Brauer group of $X$: it sits in an exact sequence
\begin{equation}\label{eq0}
0\to \NS(X)\otimes\Q/\Z\to H^2_\et(X,\Q/\Z(1))\to \Br(X)\to 0.
\end{equation}

Here we write $A(n)$ for $\colim_{(m,p)=1}\allowbreak {}_m A\otimes \mu_m^{\otimes n}$ for a prime-to-$p$ torsion abelian group $A$, and we set for $n\ge 0$, $i\in\Z$: 
\[H^i(X,\Q_p/\Z_p(n))=\colim_s H^{i-n}_\et(X,\nu_s(n))\]
where $p$ is the exponential characteristic of $k$ and, if $p>1$,  $\nu_s(n)$ is the $s$-th sheaf of logarithmic Hodge-Witt differentials of weight $n$ \cite{drW,milne,gros-suwa}. (See \cite[p. 629, (5.8.4)]{drW} for the $p$-primary part in characteristic $p$ in \eqref{eq0}.)

\begin{thm}\label{cr2}  
There are natural isomorphisms
\begin{align*}
\beta':\Br(X)\{p'\}(1)&\iso H_\ind^3(X,\Z(2))\{p'\}\\
\beta_p:H^2(X,\Q_p/\Z_p(2))&\iso H_\ind^3(X,\Z(2))\{p\}
\end{align*}
where $\{p\}$ (resp. $\{p'\}$) denotes $p$-primary torsion (resp. prime-to-$p$ torsion.) 
\end{thm}

Theorem \ref{cr2} gives an interpretation of the Brauer group (away from $p$)\footnote{The group $H^2(X,\Q_p/\Z_p(2))$ is very different from $\Br(X)\{p\}$: suppose that $k$ is the algebraic closure of a finite field $\F_q$ over which $X$ is defined. In \cite[Rk 5.6]{milne}, Milne proves
\[\det(1-\gamma t\mid H^i(X,\Q_p(n)) = \prod_{v(a_{ij})=v(q^n)} (1-(q^n/a_{ij})t) \]
where $\gamma$ is the ``arithmetic'' Frobenius of $X$ over $\F_q$ and the $a_{ij}$ are the eigenvalues of the ``geometric'' Frobenius  acting on the crystalline cohomology $H^i(X/W)\otimes \Q_p$ (or, equivalently, on $l$-adic cohomology for $l\ne p$ by Katz-Messing). We get $V_p(\Br(X)\{p\})$ for $i=2,n=1$ and $V_p(H^2(X,\Q_p/\Z_p(2)))$ for $i=2,n=2$.
} in terms of algebraic cycles. In view of \eqref{eq0}, it also implies:

\begin{cor}\label{c1} If $b_2-\rho>0$, $H_\ind^3(X,\Z(2))$ is infinite. In
  characteristic zero, if $p_g>0$ then $ H_\ind^3(X,\Z(2))$ is infinite.\qed
\end{cor}

To my knowledge, this is the first general result on indecomposable (2,1)-cycles.
It relates to the following open question:

\begin{qn}[See also Remark \protect{\ref{r1}}]\label{q1} Is there a surface $X$ such that $b_2-\rho>0$ but $H_\ind^3(X,\Z(2))\otimes \Q\allowbreak= 0$?
\end{qn}

Many examples of complex surfaces $X$ for which $H_\ind^3(X,\Z(2))$ is not torsion have been given,  see e.g. \cite{cdkl} and the references therein. In most of them, one shows that a version of the Beilinson regulator with values in a quotient of Deligne cohomology takes non torsion values on this group. On the other hand, there are examples of complex surfaces $X$ with $p_g>0$ for which the regulator vanishes rationally \cite[Th. 1.6]{voisin}, but there seems to be no such $X$ for which one can decide whether $H_\ind^3(X,\Z(2))\otimes \Q\allowbreak= 0$.

Question \ref{q1} evokes Mumford's nonrepresentability theorem for the Albanese kernel $T(X)$ in the Chow group $CH_0(X)$ under the given hypothesis. It is of course  much harder, but not unrelated. The link comes through the \emph{transcendental part of the Chow motive of $X$}, introduced and studied in \cite{kmp}. If we denote this motive by $t_2(X)$ as in loc. cit., we have
\[T(X)_\Q= \Hom_\Q(t_2(X),\L^2) = H^4(t_2(X),\Z(2))_\Q\]
\cite[Prop. 7.2.3]{kmp}. Here, all groups are taken in the category $\Ab\otimes \Q$ of abelian groups modulo groups of finite exponent and $\Hom_\Q$ denotes the refined $\Hom$ group on the category $\sM_\rat^\eff(k,\Q)$ of effective Chow motives with $\Q$ coefficients (see Section \ref{loc} for all this), while $\L$ is the Lefschetz motive; to justify the last term, note that Chow correspondences act on motivic cohomology, so that motivic cohomology of a Chow motive makes sense. We show:

\begin{thm}[see Proposition \protect{\ref{pr3}}]\label{t2} If $X$ is a surface, we have an isomorphism in $\Ab\otimes \Q$:
\[H_\ind^3(X,\Z(2))_\Q\simeq H^3(t_2(X),\Z(2))_\Q.\]
\end{thm}

\begin{cor}[\protect{\cite[Prop. 2.15]{ctr}}]\label{c2} In Theorem \ref{t2}, assume that $k$ has infinite transcendence degree over its prime subfield. If $T(X)=0$, then $H_\ind^3(X,\Z(2))$ is finite.
\end{cor}

\begin{proof}  Under the hypothesis on $k$, $T(X)=0$ $\iff$ $t_2(X)=0$ \cite[Cor. 7.4.9 b)]{kmp}. Thus, $T(X)=0$ $\Rightarrow H_\ind^3(X,\Z(2))_\Q=0$ by Theorem \ref{t2}. This means that $H_\ind^3(X,\Z(2))$ has finite exponent, hence is finite by Theorem \ref{cr2} and the known structure of $\Br(X)$.
\end{proof}

\begin{rque}\label{r1}1)  For $l\ne p$, $H_\ind^3(X,\Z(2))\{l\}$ finite $\iff$ $b_2-\rho=0$ by Theorem \ref{cr2}.  Under Bloch's conjecture, this implies $t_2(X)=0$ \cite[Cor. 7.6.11]{kmp}, hence $T(X)=0$ and (by Theorem \ref{t2}) $H_\ind^3(X,\Z(2))$ finite. This provides conjectural converses to Corollaries \ref{c1} (for a surface) and \ref{c2}.\\ 
2) The quotient of $H_\ind^3(X,\Z(2))_\tors$ by its maximal divisible subgroup is dual to $\NS(X)_\tors$, at least away from $p$: we leave this to the interested reader.
\end{rque}

In Section \ref{sgen}, we apply Theorem \ref{t2} to give   a proof of Ro\v\i t\-man's theorem that $T(X)$ is uniquely divisible, up to a group of finite exponent. This proof is related to Bloch's \cite{bloch}, but avoids Lefschetz pencils; we feel that $t_2(X)$ gives a new understanding of the situation.

\subsection*{Acknowledgements} This work was done during a visit in the Tata Institute of Fundamental research (Mumbai) in the fall 2006: I would like to thank R. Sujatha for her invitation, TIFR for its hospitality and support and IFIM for travel support. I also thank James Lewis, Joseph Oesterl\'e and Masanori Asakura for helpful remarks. Finally, I thank the referee for insisting on more details in the proof of Proposition \ref{pr2}, which helped to uncover a gap now filled by Lemma \ref{lmanquant}.

\section{Proof of Theorem \ref{cr2}}\label{s1}

This proof is an elaboration of the arguments of Colliot-Th\'el\`ene and Raskind in \cite{ctr}, completed by Gros-Suwa \cite[Ch. IV]{gros-suwa} for $l=\car k$. We use motivic cohomology as it smoothens the exposition and is more inspirational, but stress that these ideas go back to  \cite{bloch,panin,ctr} and \cite{gros-suwa}. We refer to \cite[\S 2]{cycle-etale} for an exposition of ordinary and \'etale motivic cohomology and the facts used below, especially to \cite[Th. 2.6]{cycle-etale} for the comparison with \'etale cohomology of twisted roots of unity and logarithmic Hodge-Witt sheaves.

Multiplication by $l^s$ on \'etale motivic cohomology yields ``Bockstein'' exact sequences
\[0\to H^i_\et(X,\Z(n))/l^s\to H^i_\et(X,\Z/l^s(n))\to
{}_{l^s}H^{i+1}_\et(X,\Z(n))\to 0\]
for any prime $l$, $s\ge 1$, $n\ge 0$ and $i\in\Z$.
Since $\lim^1 H^i_\et(X,\Z(n))/l^s=0$, one gets in the 
limit exact sequences:
\begin{equation}\label{eq3}
0\to H^i_\et(X,\Z(n))^{\widehat{}}\by{a} H^{i}_\et(X,\hat{\Z}(n))\by{b}
\hat{T}(H^{i+1}_\et(X,\Z(n)))\to 0
\end{equation}
where $\hat{T}(-)= \Hom(\Q/\Z,-)$
denotes the total Tate module. This first yields:

\begin{prop}\label{p1}
For $i\ne 2n$, $\IM a\otimes \Z[1/p]$ is finite in \eqref{eq3}$\otimes \Z[1/p]$ and $H^i_\et(X,\Z(n))\allowbreak\otimes \Z[1/p]$ is an extension of a finite group
by a divisible group. If $p>1$, $H^i_\et(X,\Z(n))\allowbreak\otimes \Z_{(p)}$ is an extension of a group of finite exponent by a divisible group, and is divisible if $i=n$. In particular, $H^n_\et(X,\Z(n))$ is an extension of a finite group of order prime to $p$ by a divisible group.
\end{prop}

\begin{proof} This is the argument of \cite[1.8 and 2.2]{ctr}. Let us summarise it: $H^i_\et(X,\Z(n))$ is ``of weight $0$'' and $H^{i}_\et(X,\hat{\Z}(n))$ is ``of weight $i-2n$'' by Deligne's proof of the Weil conjectures. It follows that $a\otimes \Z[1/p]$ has finite image in every $l$-component, hence has finite image by  Gabber's theorem \cite{gabber}. One derives the structure of $H^i_\et(X,\Z(n))\otimes \Z[1/p]$ from this. 

On the referee's request, we add more details. Since $X$ is defined over a finitely generated field, motivic cohomology commutes with filtering inverse limits of smooth schemes (with affine transition morphisms) and $l$-adic cohomology is invariant under algebraically closed extensions, to show that $a$ has finite image we may assume that $k$ is the algebraic closure of a finitely generated field $k_0$ over which $X$ is defined. If $i\ne 2n$ and $l\ne p$, then $H^{i}_\et(X,\Z_l(n))^U$ is finite for any open subgroup $U$ of $Gal(k/k_0)$ \cite[1.5]{ctr}, while $H^i_\et(X,\Z(n))=\bigcup_U H^i_\et(X,\Z(n))^U$. Thus the image $I(l)$ of the composition $H^i_\et(X,\Z(n))\to H^i_\et(X,\Z(n))^{\widehat{}}_l\by{a_l} H^{i}_\et(X,\Z_l(n))$ is contained in the (finite) torsion subgroup of \break $H^{i}_\et(X,\Z_l(n))$, hence this composition factors through $H^i_\et(X,\Z(n))/l^s$ for $s\gg 0$, implying that $\IM a_l=I(l)$ is finite, and $0$ for almost all $l$ by \cite{gabber}. The conclusion now follows by Lemma \ref{ctr} below. 

If $l=p$, the group $H^{i}_\et(X,\Q_p(n))^U$ is still $0$ for $i\ne 2n$ by  \cite[II.2.3]{gros-suwa}. The group $H^{i}_\et(X,\Z_p(n))$  has the structure of an extension of a pro-\'etale group by a unipotent quasi-algebraic group by \cite[Th. 3.3 b)]{il-ra}, hence has finite exponent independent of $k$. Therefore $H^{i}_\et(X,\Z_p(n))^U$ has bounded exponent when $U$ varies, hence (as above) $\IM a_p$ has finite exponent, and the first claim. For the second one, $H^{n}_\et(X,\Z_p(n))$ is always torsion-free by \cite[Ch. II, Cor. 2.17]{drW}.
\end{proof}

\begin{lemme}\label{ctr} Let $A$ be an abelian group such that $\hat{A}=\lim A/m$ has finite exponent. Then $A$ is an extension of $\hat{A}$ by a divisible group.
\end{lemme}

\begin{proof} This is the argument of \cite[Th. 1.8]{ctr}, that we reproduce here. First, $\hat{A}\iso A/m_0$ for some $m_0\ge 1$, hence $A\to \hat{A}$ is surjective. Now $A/m\iso A/m_0$ for any multiple $m$ of $m_0$, hence $\Ker(A\to \hat{A})=mA$ for any such $m$: thus this kernel is divisible as claimed.
\end{proof}

\begin{rque} In characteristic $p$, the torsion subgroup of $H^{i}_\et(X,\Z_p(n))$ may well be infinite for $i>n$ (compare \cite[Ch. II, \S 7]{drW}), and then so is the quotient of $H^i_\et(X,\Z(n))\otimes \Z_{(p)}$ by its maximal divisible subgroup.
\end{rque}


Consider now the case $n=2$. Recall that $H^i(X,\Z(2))\iso H^i_\et(X,\Z(2))$ for $i\le 3$ from the Merkurjev-Suslin theorem (cf. \cite[(2-6)]{cycle-etale}). 

For $l\ne p$,  let
\begin{align*}
H^2_\ind(X,\mu_{l^n}^{\otimes 2})& = \Coker(\Pic(X)\otimes \mu_{l^n}\to
H^2_\et(X,\mu_{l^n}^{\otimes 2}))\\
H^2_\ind(X,\Z_l(2))& = \Coker(\Pic(X)\otimes \Z_l(1)\to
H^2_\et(X,\Z_l(2))).
\end{align*}

\begin{lemme}\label{l1} For $l\ne p$, there is a canonical isomorphism $H^2_\ind(X,\Z_l(2))\allowbreak\simeq T_l(\Br(X))(1)$. 
In particular, this group is torsion-free.
\end{lemme}

\begin{proof} Straightforward from the Kummer exact sequence.
\end{proof}

We have a  commutative diagram 
\begin{equation}\label{eq4}
\begin{CD}
0&\longrightarrow& \Pic(X)\otimes \mu_{l^s}&\longrightarrow&H^2_\et(X,\mu_{l^s}^{\otimes 2})&\longrightarrow& H^2_\ind(X,\mu_{l^s}^{\otimes 2})&\longrightarrow& 0\\
&&@V{\text{surjective}}VV @V\alpha_s VV\\
0&\longrightarrow& {}_{l^s}(\Pic(X)\otimes k^*)&\longrightarrow& {}_{l^s}H^3(X,\Z(2))&\longrightarrow& {}_{l^s}
  H_\ind^3(X,\Z(2)) &\longrightarrow& 0
\end{CD}
\end{equation}
where the upper row is exact and the lower row is a
complex. This diagram is equivalent to the one in  \cite[2.8]{ctr}, but the proof of its commutativity is easier, as a consequence of the compatibility of Bockstein boundaries with cup-product in hypercohomology. This yields maps
\begin{align}
H^2_\ind(X,\mu_{l^s}^{\otimes 2})&\by{\beta_s} {}_{l^s}
  H_\ind^3(X,\Z(2)),\label{1}
\end{align}
an inverse limit commutative  diagram
\begin{equation}\label{eq5}
\begin{CD}
0&\to& \NS(X)\otimes \Z_l(1)&\longrightarrow& H^2_\et(X,\Z_l(2))&\overset{\pi}{\longrightarrow}& H^2_\ind(X,\Z_l(2))&\to& 0\\
&&@V{\text{surjective}}VV @V\hat{\alpha}VV @V\hat{\beta} VV\\
0 &\to& T_l(\Pic(X)\otimes k^*)&\longrightarrow& T_l(H^3(X,\Z(2))&\longrightarrow& T_l(  H_\ind^3(X,\Z(2)) &\to& 0
\end{CD}
\end{equation}
(note that $\Pic(X)\otimes \mu_{l^s}\allowbreak\iso \NS(X)\otimes \mu_{l^s}$) and a direct limit commmutative  diagram
\begin{equation}\label{eq5ind}
\begin{CD}
0&\to&\Pic(X)\otimes \mu_{l^\infty}&\longrightarrow& H^2(X,\Q_l/\Z_l(2))&\longrightarrow& \Br(X)\{l\}(1)&\to& 0\\
&&@V{\wr}VV @V{\alpha_l}VV @V{\beta_l}VV\\
0&\to&(\Pic(X)\otimes k^*)\{l\}&\longrightarrow& H^3(X,\Z(2))\{l\}&\longrightarrow& H_\ind^3(X,\Z(2))\{l\}&\to& 0
\end{CD}
\end{equation}
where $\beta_l$ defines the map $\beta'$ in Theorem \ref{cr2}. Note that the left vertical map in \eqref{eq5ind} is injective because $\Tor(\Pic(X),k^*\otimes \Z[1/l])\{l\}=0$.

\begin{lemme}\label{lmanquant} If $X$ is defined over a subfield $k_0$ with algebraic closure $k$, the map $\pi$ of \eqref{eq5} has a $G$-equivariant section after $\otimes \Q$, where $G=Gal(k/k_0)$. In particular, if $k_0$ is  finitely generated, then $H^2_\ind(X,\Q_l(2))^U\allowbreak =0$ for any open subgroup $U$ of $G$. 
\end{lemme}

\begin{proof} Let $d=\dim X$: we may assume $d>1$. If $d=2$, the perfect Poincar\'e pairing $H^2_\et(X,\Q_l(1))\times H^2_\et(X,\Q_l(1))\to \Q_l$ restricts to the perfect intersection pairing $ \NS(X)\otimes {\Q_l}\otimes  \NS(X)\otimes {\Q_l}\to \Q_l$; the promised section is then given by the orthogonal complement of $ \NS(X)\otimes \Q_l(1)$ in $H^2_\et(X,\Q_l(2))$. If $d >2$, let $L\in H^2(X,\Q_l)$ be the class of a smooth hyperplane section defined over $k_0$. The hard Lefschetz theorem and Poincar\'e duality provide a perfect pairing on $H^2_\et(X,\Q_l(1))$:
\[(x,y)\mapsto x\cdot L^{d-2}\cdot y\]
which restricts to a similar pairing on $\NS(X)\otimes {\Q_l}$. The Hodge index theorem for divisors \cite[Prop. 7.4 p. 665]{kleiman} implies that the latter pairing is also nondegenerate, so we get the desired section in the same way. The last claim now follows from the vanishing of $H^2(X,\Q_l(2))^U$, see proof of Proposition \ref{p1}.
\end{proof}

We shall use the following fact, which is proven in   \cite[2.7]{ctr} (and could be reproven here with motivic cohomology in the same fashion):

\begin{lemme} \label{l0} In \eqref{eqtheta}, $N:=\Ker\theta$ has no $l$-torsion.
\end{lemme}

\begin{prop}[cf. \protect{\cite[Rk. 2.13]{ctr}}]\label{pr2} $\beta_s$ is surjective in \eqref{1} and $\hat{\beta}$ is bijective in   \eqref{eq5}; $N$ is uniquely divisible; the lower row of \eqref{eq5ind} is exact and $\beta_l$ is bijective.
\end{prop}

\begin{proof} 
Since $\Pic(X)\otimes k^*$ is $l$-divisible, Lemma \ref{l0} yields exact sequences
\begin{gather}
0\to {}_{l^s}(\Pic(X)\otimes k^*)\to {}_{l^s}A\to N/l^s\to 0\label{3}\\
0\to {}_{l^s}A\to {}_{l^s}
  H^3(X,\Z(2))\to {}_{l^s}
  H_\ind^3(X,\Z(2))\to 0\label{4}
\end{gather}
where $A=\IM\theta$, and \eqref{4} implies the surjectivity of $\beta_s$, hence of 
$\hat{\beta}$ since the groups $H^2_\ind(X,\mu_{l^s}^{\otimes 2})$
 are finite. Since $\alpha_s$ is surjective in \eqref{eq4}, we also get that all groups in \eqref{3} and \eqref{4} are finite. Now the upper row of \eqref{eq5} is exact; in its lower row, the homology at
$T_l(H^3(X,\Z(2))$ is  isomorphic to $N^{\widehat{}}_l$ by taking the inverse limit of  
\eqref{3} and \eqref{4}. A snake chase then yields an exact sequence 
\[H^2(X,\Z(2))^{\widehat{}}_l\simeq \Ker \hat{\alpha}\to \Ker \hat{\beta}\to N^{\widehat{}}_l\to 0\] 
where $\Ker \hat{\alpha}$ is finite by Proposition \ref{p1}.

If,  as in the proof of Proposition \ref{p1},  $k$ is the algebraic closure of a finitely generated field $k_0$ over which $X$ is defined and $U$ is an open subgroup of $Gal(k/k_0)$, we have an isomorphism
\[(\Ker \hat{\beta})^U\otimes \Q\iso (N^{\widehat{}}_l)^U\otimes \Q.\]

On the one hand, $(\Ker \hat{\beta})^U\otimes \Q=0$ by Lemma \ref{lmanquant} because $\Ker \hat{\beta}$ is a subgroup of $H^2_\ind(X,\Z_l(2))$; on the other hand, since $N/l$ is finite,
\[N^{\widehat{}}_l= \bigcup_U (N^{\widehat{}}_l)^U.\]

Indeed, a finite set of generators $\{n_i\}$ of $N$ modulo $lN$ also generates $N$ modulo  $l^sN$ for all $s\ge 1$, and an open subgroup $U$ of $G$ fixing all the $n_i$ also fixes $N^{\widehat{}}_l$ (so the union is in fact stationary).

This gives $N^{\widehat{}}_l\otimes \Q=0$, hence $N^{\widehat{}}_l=0$ by Lemma \ref{l0}; thus $\Ker \hat{\beta}$ is finite, hence $0$ by Lemma \ref{l1}. This also shows the $l$-divisibility of $N$, which thanks to \eqref{3} and \eqref{4} implies  the exactness of the lower row of  \eqref{eq4}, hence of \eqref{eq5ind}. Now  $\alpha_l$ is surjective, and also injective  since $\Ker\alpha_l\simeq H^2(X,\Z(2))\otimes\Q_l/\Z_l$ is $0$ by Proposition \ref{p1}. Hence $\beta_l$ is bijective.  
\end{proof}

The case of $p$-torsion is similar and easier: by Proposition \ref{p1}, we have an isomorphism
\[H^2(X,\Q_p/\Z_p(2))\iso H^3(X,\Z(2))\{p\}\]
and $H^3(X,\Z(2))\{p\}\iso H_\ind^3(X,\Z(2))\{p\}$ since $k^*$ is uniquely
$p$-divi\-si\-ble, hence also $\Pic(X)\otimes k^*$. This concludes the proof of Theorem \ref{cr2}.

\section{Refined $\Hom$ groups} \label{loc}

Let $\sA$ be an additive category; write $\sA\otimes \Q$ for the category
with the same objects as $\sA$ and $\Hom$ groups tensored with $\Q$, and $\sA\boxtimes\Q$ for the pseudo-abelian envelope of $\sA\otimes \Q$. If $\sA$ is abelian, then $\sA\otimes \Q=\sA\boxtimes \Q$ is still abelian and is the localisation of $\sA$ by the Serre subcategory $\sA_\tors$ of objects $A$ such that $n1_A=0$ for some integer $n>0$ (e.g. \cite[Prop. B.3.1]{BVK}).

For $\sA=\Ab$, the category of abelian groups, one has a natural functor ``tensoring objects with $\Q$''
\[\Ab\otimes \Q\to \Vec_\Q\]
to $\Q$-vector spaces. This functor is fully faithful when restricted to the full subcategory of $\Ab\otimes \Q$ given by finitely generated abelian groups, but for example it does not send $\Q/\Z$ to $0$. For clarity, we shall write
\begin{equation}\label{eqQ}
A_\Q, \quad A\otimes \Q
\end{equation}
for the image of an abelian group $A\in \Ab$ respectively in $\Ab\otimes \Q$ and $\Vec_\Q$.

 Let $F$ be an additive functor (covariant or contravariant) from $\sA$ to $\Ab$, the category of abelian groups: it then induces a functor
\[F_\Q:\sA\boxtimes \Q\to \Ab\otimes \Q.\]

In particular, we get a bifunctor
\[\Hom_\Q:(\sA\boxtimes \Q)^\op\times \sA\boxtimes\Q\to \Ab\otimes \Q\]
which refines the bifunctor $\Hom$ of $\sA\boxtimes\Q$.

We shall apply this to $\sA=\sM_\rat^\eff(k)$, the category of effective Chow motives with integral coefficients: the category  $\sA\boxtimes \Q$ is then equivalent to the category  $\sM_\rat^\eff(k,\Q)$ of Chow motives with rational coefficients.

\section{Chow-K\"unneth decomposition of $\sK_2$-cohomology}

In this section,  $X$ is a connected  surface. Its Chow motive
$h(X)\in \sM_\rat^\eff(k,\Q)$ then enjoys a refined Chow-K\"unneth decomposition
\begin{equation}\label{eq1}
h(X)=h_0(X)\oplus h_1(X)\oplus h_2^\alg(X)\oplus t_2(X)\oplus h_3(X)\oplus
h_4(X)
\end{equation}
\cite[Prop. 7.2.1 and 7.2.3]{kmp}. The projectors defining this decomposition act on 
the groups $H^i(X,\Z(2))_\Q$; we propose to compute the corresponding direct summands 
$H^i(M,\Z(2))_\Q$. To be more concrete, we shall express this in terms of the $\sK_2$-cohomology of $X$.

We keep the notation 
\[H_\ind^1(X,\sK_2) = \Coker(\Pic(X)\otimes k^*\to
H^1(X,\sK_2))\]
to which we adjoin
\[H_\ind^0(X,\sK_2) = \Coker(K_2(k)\to
H^0(X,\sK_2)).\]

To relate with the notation in Section \ref{s1}, recall that $H^2(k,\Z(2))\allowbreak=K_2(k)$ and $H^2(X,\Z(2))=H^0(X,\sK_2)$.

We shall also need a  smooth connected hyperplane section $C$ of $X$, appearing in the construction of \eqref{eq1} \cite{murre,scholl}, and its own 
Chow-K\"unneth decomposition attached to the choice of a rational  point:
\begin{equation}\label{eq2}
h(C)=h_0(C)\oplus h_1(C)\oplus h_2(C).
\end{equation}

The projectors defining \eqref{eq2} have integral coefficients, while those 
defining \eqref{eq1} only have  rational coefficients in general.

The following proposition extends the computations of \cite[7.2.1 and 7.2.3]{kmp} to weight $2$ motivic cohomology.

\begin{prop}\label{pr3} a) We have the following table for $H^i(M,\Z(2))$: 
\begin{center}
\begin{tabular}{| c | c | c | c | c  |}
\hline
$M=$& $h_0(C)$ & $h_1(C)$ &$h_2(C)$ \\
\hline
$i=2$&$K_2(k)$ &$H_\ind^0(C,\sK_2)$ &$0$ \\
$i=3$&$0$ &$V(C)$ &$k^*$\\
$i>3$&$0$ &$0$ &$0$\\
\hline
\end{tabular}
\end{center}
where $V(C)=\Ker(H^1(C,\sK_2)\by{N}k^*)$ is Bloch's group.\\
b) We have the following table for $H^i(M,\Z(2))$, where all groups are taken in $\Ab\otimes \Q$ (see Section \ref{loc}): 
\begin{center}
\begin{tabular}{| c | c | c | c | c  | c | c |c |}
\hline
$M=$& $h_0(X)$ & $h_1(X)$ &$h_2^\alg(X)$ &$t_2(X)$ &$h_3(X)$ &$h_4(X)$ \\
\hline
$i=2$&$K_2(k)$ &$A$ &$0$ &$B$ &$0$&$0$\\
$i=3$&$0$ &$\Pic^0(X)k^*$ &$\NS(X)\otimes k^*$ &$H_\ind^1(X,\sK_2)$&$0$&$0$\\
$i=4$&$0$ &$0$ &$0$&$T(X)$ &$\Alb(X)$&$\Z$\\
$i>4$&$0$ &$0$ &$0$ &$0$&$0$&$0$\\
\hline
\end{tabular}
\end{center}
where 
\begin{align*}
\Pic^0(X)k^*&=\IM(\Pic^0(X)\otimes k^*\to H^1(X,\sK_2))\\
A &=\IM(H_\ind^0(X,\sK_2)\to H_\ind^0(C,\sK_2))\\
B &=\Ker(H_\ind^0(X,\sK_2)\to H_\ind^0(C,\sK_2)).
\end{align*}
\end{prop}

\begin{proof} We proceed by exclusion as in the proof of \cite[Th. 7.8.4]{kmp}.
Let us start with a). We use the notation \eqref{eqQ} of Section \ref{loc}.

\begin{itemize}
\item  For $i>3$, $H^{i}(M,\Z(2))_\Q$ is a direct summand of $H^{i}(C,\Z(2))_\Q=0$.
\item One has $h_2(C)=\L$, hence 
\[H^i(h_2(C),\Z(2))_\Q=H^{i-2}(k,\Z(1))_\Q=
\begin{cases}
k^*_\Q&\text{if $i=3$}\\
0&\text{else.}
\end{cases}\]
\item One has 
\[H^i(h_0(C),\Z(2))_\Q=H^i(k,\Z(2))_\Q=
\begin{cases}
K_2(k)_\Q &\text{if $i=2$}\\
0&\text{if $i>2$.}
\end{cases}
\]
\item The case of $M=h_1(C)$ follows from the two previous ones by exclusion.
\end{itemize}

Let us come to b). 

\begin{itemize}
\item For $i>4$, $H^{i}(M,\Z(2))_\Q$ is a direct summand of $H^{i}(X,\Z(2))_\Q=0$.
\item One has $h_4(X)=\L^2$, hence
\[H^i(h_4(X),\Z(2))_\Q=H^{i-4}(k,\Z)_\Q=
\begin{cases}
\Z_\Q&\text{if $i=4$}\\
0&\text{else.}
\end{cases}\]
\item One has $h_3(X)=h_1(X)(1)$, hence
\[H^i(h_3(X),\Z(2))_\Q=H^{i-2}(h_1(X),\Z(1))_\Q.\]
As $h_1(X)$ is a direct summand of $h_1(C)$, $H^{i-2}(h_1(X),\Z(1))_\Q$ is a direct summand of $H^{i-2}(C,\Z(1))_\Q$. This group is $0$ for $i\ne 3,4$. For $i=3$, one has $H^1(C,\Z(1))_\Q\allowbreak=H^1(h_0(C),\Z(1))_\Q$, hence
\[H^{1}(h_1(C),\Z(1))_\Q=H^{1}(h_1(X),\Z(1))_\Q=0.\] 
For $i=4$,  $H^2(h_1(X),\Z(1))_\Q = \Alb(X)_\Q$ (cf. Murre \cite{murre}).
\item One has $h_2^\alg(X)=\NS(X)(1)$, hence
\begin{multline*}
H^i(h_2^\alg(X),\Z(2))_\Q= (H^{i-2}(k,\Z(1))\otimes \NS(X))_\Q \\
=
\begin{cases}
(\NS(X)\otimes k^*)_\Q &\text{if $i=3$}\\
0&\text{else.}
\end{cases}
\end{multline*}
\item One has 
\[H^i(h_0(X),\Z(2))_\Q=H^i(k,\Z(2))_\Q=
\begin{cases}
K_2(k)_\Q &\text{if $i=2$}\\
0&\text{if $i>2$.}
\end{cases}
\]
\item As $h^1(X)$ is a direct summand of $h^1(C)$, $H^i(h^1(X),\Z(2))_\Q$ is a direct summand of 
$H^i(C,\Z(2))_\Q$: this group is therefore $0$ $i>3$. This completes row $i=4$ by 
exclusion.
\item The action of refined Chow-K\"unneth  projectors respects the homomorphism
$(\Pic(X)\otimes k^*)_\Q\to H^3(X,\Z(2))_\Q$. As the action of $\pi_2^\tr$ (defining
$t_2(X)$) is $0$ on $\Pic(X)_\Q$, we get $H^3(t_2(X),\Z(2))_\Q\allowbreak\simeq H_\ind^1(X,\sK_2)_\Q$,
which completes row $i=3$ by exclusion.
\item The construction of $\pi_2^\tr$ \cite[proof of 
2.3]{kmp} shows that the composition
\[h(C)\overset{i_*}{\to} h(X)\to t_2(X)\]
is $0$. Hence the composition
\[H^i(t_2(X),\Z(2))_\Q\to H^i(X,\Z(2))_\Q\overset{i^*}{\to} H^i(C,\Z(2))_\Q\]
is $0$ for all $i$. Applying this for $i=2$, we see that $H^2(t_2(X),\Z(2))_\Q\allowbreak\subseteq
B_\Q$. On the other hand, $H^2(h_1(X),\Z(2))_\Q$ is a direct summand of $H^2(h_1(C),\Z(2))_\Q$, hence injects in  $A_\Q$. By exclusion, we have $H^2(t_2(X),\Z(2))_\Q\oplus H^2(h_1(X),\Z(2))_\Q\simeq H_\ind^0(X,\Z(2))_\Q$,
hence row $i=2$.
\end{itemize}
\end{proof}

\begin{rque} Let us clarify the ``reasoning by exclusion'' that has been used repeatedly in this proof. Let $F$ be a functor from smooth projective varieties to $\Ab\otimes\Q$, provided with an action of Chow correspondences. Then $F$ automatically extends to $\sM_\rat^\eff(k,\Q)$, and we wish to compute the effect of a Chow-K\"unneth decomposition of $h(X)$ on $F(X)$. The reasoning above is as follows in its simplest form:

Suppose that we have a motivic decomposition $h(X)=M\oplus M'$, hence a decomposition $F(X)=F(M)\oplus F(M')$. Suppose that we know an exact sequence
\[0\to A\to F(X)\to B\to 0\]
and an isomorphism $F(M)\simeq A$. Then $F(M')\simeq B$.

Of course this reasoning is incorrect as it stands; to justify it, one should check that if $\pi$ is the projector with image $M$ yielding the decomposition of $h(X)$, then $F(\pi)$  does have image $A$. This can be checked in all cases of the above proof, but such a verification would be tedious, double the length of the proof and probably make it unreadable. I hope the reader will not disagree with this expository choice.
\end{rque}

\section{Generalisation}\label{sgen}

In this section, we take the gist of the previous arguments. For convenience we pass from effective Chow motives $\sM_\rat^\eff(k,\Q)$ to all Chow motives $\sM_\rat(k,\Q)$. Since \'etale motivic cohomology has an action of Chow correspondences and verifies the projective bundle formula, it yields well-defined contravariant functors
\[H^i_\et:\sM_\rat(k,\Q)\to \Ab\otimes \Q\]
such that $H^i_\et(X,\Z(n))_\Q=H^{i-2n}_\et(h(X)(-n))$ for any smooth projective $k$-variety $X$ and $i,n\in\Z$. We also have (contravariant) realisation functors
\[H^i_l:\sM_\rat(k,\Q)\to  \sC_l\otimes \Q\]
extending $l$-adic cohomology for $l\ne \car k$, where $\sC_l$ denotes the category of $l\Z$-adic inverse systems of abelian groups \cite[V.3.1.1]{SGA5}. For $l=\car k$ we use logarithmic Hodge-Witt cohomology as in Theorem \ref{cr2} \cite[\S 2]{milne}, \cite{gros-suwa}.

\begin{defn} Let $M\in \sM_\rat(k,\Q)$.  If $i\in \Z$, we say that \emph{$M$ is pure of weight $i$} if $H^j_l(M)=0$ for all $j\ne i$ and all primes $l$. 
\end{defn}

For example, if $h(X)=\bigoplus_{i=0}^{2d} h_i(X)$ is a Chow-K\"unneth decomposition of the motive $h(X)$ of a $d$-dimensional smooth projective variety $X$, then $h_i(X)$ is pure of weight $i$. If $d=2$, the motive $t_2(X)(-2)$ is pure of weight $-2$ as a direct summand of $h_2(X)(-2)$.

\begin{thm} \label{t3}Let $M$ be pure of weight $i$. Then $H^j_\et(M)$ is uniquely divisible for $j\ne i,i+1$. If moreover $i\ne 0$, then $H^i_\et(M)$  is uniquely divisible and $H^{i+1}_\et(M)\{l\}\simeq H^i_l(M)\otimes \Q/\Z$.
\end{thm}

(An object $A\in \Ab\otimes \Q$ is \emph{uniquely divisible} if multiplication by $n$ is an automorphism of $A$ for any integer $n\ne 0$.)

\begin{proof} As in Section \ref{s1}, we have Bockstein exact sequences in $\sC_l\otimes \Q$
\[0\to H^j_\et(M)/l^*\by{a} H^j_l(M)\to {}_{l^*}H^{j+1}_\et(M)\to 0\]
which yields the first statement. For the second one, the weight argument of \cite{ctr} (developed in the proof of Proposition \ref{p1} above) yields $\IM a=0$. 
\end{proof}

Let $X$ be a surface. Applying Theorem \ref{t3} to $M=t_2(X)(-2)$ as above, we get that $H^i_\et(t_2(X),\Z(2))$ is uniquely divisible for $i\ne 3$ and 
\[H^3_\et(t_2(X),\Z(2))\{l\}\allowbreak\simeq H^3_\tr(X,\Z_l(2)\otimes \Q/\Z \simeq \Br(X)\{l\}\]
in $\Ab\otimes \Q$, recovering a slightly weaker version of Theorem \ref{cr2} in view of Proposition \ref{pr3}. For $i=4$, the exact sequence \cite[(2-7)]{cycle-etale}
\begin{equation*}\label{eqlk}
0\to CH^2(X)\to H^4_\et(X,\Z(2))\to H^0(X,\sH^3_\et(\Q/\Z(2)))\to 0
\end{equation*}
shows that $CH^2(X)\iso H^4_\et(X,\Z(2))$ since $\dim X=2$, whence 
\[T(X)=H^4(t_2(X),\Z(2))\iso H^4_\et(t_2(X),\Z(2))\] 
yielding a proof of Ro\v\i tman's theorem up to small torsion.

\begin{rque} This argument is not integral because the projector $\pi_2^\tr$ defining $t_2(X)$ is not an integral correspondence.  It is however $l$-integral for any $l$ prime to a denominator $D$ of $\pi_2^\tr$. This $D$ is essentially controlled by the degree of the Weil isogeny
\[\Pic^0_{X/k}\to \Pic^0_{C/k}=\Alb(C)\to \Alb(X)\]
where $C$ is the ample curve involved in the construction of $\pi_2^\tr$. If one could show that various $C$'s can be chosen so that the  corresponding degrees have gcd equal to $1$, one would deduce a full proof of Ro\v\i tman's theorem from the above.
\end{rque}


\begin{thebibliography}{II}
\bibitem{BVK} L. Barbieri-Viale, B. Kahn {\it On the derived category of $1$-motives}, \url{http://arxiv.org/abs/1009.1900}, to appear in Ast\'erisque. 
\bibitem{bloch} S. Bloch {\it Torsion algebraic cycles and a theorem of Roitman}, Compositio Math. {\bf 39} (1979), 107--127.
\bibitem{cdkl} X. Chen, C. Doran, M. Kerr, J. Lewis {\it Normal functions, Picard-Fuchs equations, and elliptic fibrations on K3 surfaces}, \url{http://arxiv.org/abs/1108.2223}, to appear in J. reine angew. Math. 
\bibitem{ctr} J.-L. Colliot-Th\'el\`ene, W. Raskind {\it
  $\sK_2$-cohomology and the second Chow group}, Math. Ann. {\bf 270}
  (1985), 165--199.
\bibitem{gabber} O. Gabber {\it Sur la torsion dans la cohomologie
  $l$-adique d'une vari\'et\'e}, C. R. Acad. Sci. Paris {\bf 297}
  (1983), 179--182.
\bibitem{gros-suwa} M. Gros, N. Suwa {\it Application d'Abel-Jacobi $p$-adique et cycles
alg\'ebriques}, Duke Math. J. {\bf 57} (1988),  579--613.
\bibitem{drW} L. Illusie {\it Complexe de de Rham-Witt et cohomologie cristalline}, Ann. Sci. \'ENS {\bf 12} (1979), 501--661.
\bibitem{il-ra} L. Illusie, M. Raynaud {\it Les suites spectrales associ\'ees au complexe de de Rham-Witt}, Publ. Math. IH\'ES {\bf 57} (1983), 73--212.
\bibitem{SGA5} J.-P. Jouanolou {\it Syst\`emes projectifs $l$-adiques}, {\it in} Cohomologie $l$-adique et fonctions $L$, S\'eminaire de G\'eom\'etrie alg\'ebrique du Bois-Marie (SGA5), A. Grothendieck et al, Lect. Notes in Math. {\bf 589}, Exp. V, Springer, 1977, 204--250.
\bibitem{kmp} B. Kahn, J. P. Murre, C. Pedrini {\it On the transcendental part of the motive
of a surface}, {\it in} Algebraic cycles and motives, LMS Series {\bf 344} (2), Cambridge University Press, 2007, 143--202.
\bibitem{cycle-etale} B. Kahn {\it Classes de cycle motiviques \'etales}, Alg. and number theory {\bf 6--7} (2012), 1369--1407.
\bibitem{kleiman} S. Kleiman {\it Les th\'eor\`emes de finitude pour le foncteur de Picard}, {\it in} Th\'eorie des intersections et th\'eor\`eme de Riemann-Roch, S\'eminaire de g\'eom\'etrie alg\'ebrique (SGA 6), dirig\'e par P. Berthelot, A. Grothendieck et L. Illusie, Exp. XIII,  616--666.
\bibitem{milne} J. S. Milne {\it Values of zeta functions of varieties over finite fields}, Amer. J. Math. {\bf 108} (1988), 297--360. 
\bibitem{murre} J. Murre {\it On the motive of an algebraic surface},
  J. Reine angew. Math. {\bf 409} (1990), 190--204.
\bibitem{panin} I. Panin, {\it Fields whose $K_2$ is $0$. Torsion in $H^1(X,\sK_2)$ and $CH^2(X)$}, Zap. LOMI {\bf 116} (1982), 1011-1046.
\bibitem{scholl} A.J. Scholl {\it Classical motives},   Proc. Symposia in Pure Math.  {\bf 55} (I) (1994), 163--187.
\bibitem{voisin} C. Voisin {\it Variations of Hodge structure and algebraic cycles}, Proc. ICM, Z\"urich, 1994.
\end{thebibliography}
\end{document}